\documentclass{article}

\usepackage{amsmath,amssymb,amsthm}
\usepackage{cite}
\usepackage{mathtools}
\usepackage{yhmath}
\usepackage[pdfpagelabels]{hyperref}
\numberwithin{equation}{section}
\allowdisplaybreaks[3]

\newtheorem{theorem}{Theorem}[section]
\newtheorem{corollary}[theorem]{Corollary}

\newtheorem{proposition}[theorem]{Proposition}
\theoremstyle{definition}

\newtheorem{assume}{Assumption}
\theoremstyle{remark}
\newtheorem{remark}[theorem]{Remark}

\DeclareMathOperator*{\esssup}{ess\,sup}

\newcommand{\R}{{\mathbb{R}}}
\newcommand{\X}{{\mathbb{R}^d}}
\newcommand{\N}{\mathbb{N}}

\renewcommand{\L}{\mathcal{L}}
\newcommand{\eps}{\varepsilon}
\newcommand{\la}{\lambda}
\newcommand{\La}{\Lambda}
\renewcommand{\a}{\alpha}
\newcommand{\Tau}{\Upsilon}
\newcommand{\al}{\underline{\a}}
\newcommand{\au}{\overline{\a}}
\newcommand{\B}{{\mathcal{B}}}
\newcommand{\K}{{\mathcal{K}}}
\newcommand{\M}{{\mathcal{M}}}
\newcommand{\Mfm}{\M^1_{\mathrm{fm}}(\Ga)}
\newcommand{\Bb}{{\B_{\mathrm{b}}(\X)}}
\newcommand{\Ga}{\Gamma}
\newcommand{\ga}{\gamma}
\newcommand{\Bbs}{B_{\mathrm{bs}}(\Ga_0)}

\title{Around Ovsyannikov's method}

\author{Dmitri Finkelshtein\thanks{Department of Mathematics,
Swansea University, Singleton Park, Swansea SA2 8PP, U.K.; Institute of Mathematics, Kiev, 01601, Ukraine ({\tt d.l.finkelshtein@swansea.ac.uk}).}}

\begin{document}

\maketitle

\begin{abstract}
We study existence, uniqueness, and a limiting behaviour of solutions to an abstract linear evolution equation in a scale of Banach spaces. The generator of the equation is a perturbation of the operator which satisfies the classical assumptions of Ovsyannikov's method by a generator of a $C_0$-semigroup acting in each of the spaces of the scale. The results are (slightly modified) abstract version of those considered in \cite{FKKozK2014} for a particular equation. An application to a birth-and-death stochastic dynamics in the continuum is considered.


\end{abstract}


\section{Introduction}
The study of Markov evolutions for distributions of points in Euclidean space may be reduced to the study of the evolution equations
\begin{equation}\label{abstreq}
  \frac{d}{dt} u(t)=Z u(t), \qquad u(0)=u_0,
\end{equation}
in the Fock-type spaces with weighted $L^1$- or $L^\infty$-norms with an (unbounded) linear operator $Z$, see, e.g., \cite{KKM2008,FKO2009,FKK2011a} and the references therein. The well-posedness of the initial value problem \eqref{abstreq} in a Banach space (existence, uniqueness, and continuous dependence on the initial value) requires the operator $Z$ being a generator of a strongly continuous ($C_0$-) semigroup of linear operators acting in this space. The semigroup approach to the evolution of the so-called birth-and-death dynamics of the distributions of points mentioned above was realized in, e.g., \cite{KKZ2006,KKM2008,FKK2009,FKK2011a}. In particular, the technical restrictions on the birth and death rates were introduced. However, for some important models these restrictions either were not satisfied, cf. \cite{FKKZ2014}, or required more strong assumptions on the birth and death rates (that is less interesting for applications), compare \cite{KKZ2006} and \cite{FKKoz2011} or \cite{FKK2009} and \cite{FKKozK2014}. Moreover, the dynamics with jumps were not covered by the semigroup approach at all, see, e.g., \cite{BKKK2011,BKK2013,FKKL2011b}.

To overcome these restrictions, the evolution \eqref{abstreq} was allowed to be considered in a scale of Banach spaces. In the latter approach the dynamics was constructed on a finite time interval $[0,T)$ only, and, for any $t\in(0,T)$, the solution $u(t)$ to \eqref{abstreq} belonged to a proper space of the scale. This was realized in \cite{FKKoz2011, BKKK2011,FKO2011a,FKO2011b,FKKO2014} using the so-called Ovsyannikov's method: it requires $Z$ being considered as a bounded operator between {\it each} two spaces $B_{\a'}\subset B_{\a''}$, of a scale $\{B_\a\}_{\a\in I}$, $I\subset \R$, of Banach spaces with the operator norm proportional to $(\a''-\a')^{-1}$. Originally this method was found by G.\,E.\,Shilov and A.\,G.\,Kostyuchenko, and it was firstly published in 1958 in the book \cite{GS1958} by I.\,M.\,Gelfand and G.\,E.\,Shilov. In particular, it was applied to the so-called Kovalevskya's system of first-order PDE. In 1960, T.\,Yamanaka generalized this approach for the case of the time dependent operator $Z(t)$. In 1965, the method was rediscovered in \cite{Ovs1965} by L.\,V.\,Ovsjannikov and was named in the book \cite{Tre1968} written by F.\,Tr\`{e}ves, who has realized the detailed analysis of the problem
\begin{equation}\label{abstreq1}
  \frac{d}{dt} u(t)=Z(t) u(t)+f(t), \qquad u(s)=u_s, \ s\geq0,
\end{equation}
with numerous applications. After that, the initial value problems \eqref{abstreq}, \eqref{abstreq1} with the operator norm estimate as above were named the `abstract Kovalevskya's systems' in the literature. (Note that the method allowed the immediate generalization for the complex values of~$t$.) The non-linear generalization was introduced by F.\,Tr\`{e}ves~\cite{Tre1970} under assumptions which were essentially simplified by L.\,Nirenberg~\cite{Nir1972} and T.\,Nishida \cite{Nis1977}. K.\,Deimling generalized the linear case to the equation
\begin{equation}\label{abstreq2}
  \frac{d}{dt} u(t)=Z(t) u(t)+f(t,u(t)), \qquad u(s)=u_s, \ s\geq0,
\end{equation}
for a some class of functions $f$. Being quite general, the conditions on $f$ guaranteed the existence of a solution to \eqref{abstreq1} only, cf., e.g., the survey~\cite{LS1994}.  For further generalizations see, e.g., \cite{Ben1982,Caf1990,Saf1995,Zub2004}.

On the other hand, in \cite{FKKozK2014}, there were considered the birth-and-death dynamics of complex (point) systems in the continuum whose generator of the evolution of correlation functions (see Section~\ref{sec:appl} below for the definitions) did not allow the singularity $(\a''-\a')^{-1}$ for the norm of $Z$ in \eqref{abstreq}. Namely, the singularity was of the order $(\a''-\a')^{-2}$ which can not be realized in the scheme of Ovsyannikov's method. Fortunately, the corresponding evolution equation could be rewritten in the from
\begin{equation}\label{abstreq3}
  \frac{d}{dt} u(t)=Z u(t)+Au(t), \qquad u(0)=u_0,
\end{equation}
where $Z$ was suitable for Ovsyannikov's method and $A$ was a generator of a contraction semigroup acting in {\it each} of the spaces of the scale. The important point is that the evolution of correlation functions was considered in the scale of $L^\infty$-type spaces (see, e.g., \cite{FKK2011a,FKK2014} for the explanation of the reasons). By \cite{Lot1985}, there is not a $C_0$-semigroup in a space isomorphic to $L^\infty$ with an unbounded generator (see also the proof of Proposition~\ref{existdyn} below). Therefore, the semigroups generated by $A$ were considered in the suitable subspaces of the spaces of the scale, using the technique of the so-called sun-dual semigroups, see e.g. \cite{EN2000,vNee1992}, which goes back to R.\,S.\,Phillips \cite{Phi1955}.
Another problem considered in \cite{FKKozK2014} was the convergence of the solutions to the equation
\begin{equation}\label{abstreq3-eps}
  \frac{d}{dt} u_\eps(t)=Z u_\eps(t)+Au_\eps(t), \qquad u(0)=u_0,
\end{equation}
to the solution to the limiting equation, in the course of the so-called Vlasov-type scaling, see, e.g., \cite{FKK2010a,FKK2011a} for details. This allowed to derive the so-called kinetic equation which approximately describes the behaviour of the density of the birth-and-death dynamics. All constructions in \cite{FKKozK2014} were done for the particular model, however, in the form useful for further generalization.

The aims of the present paper are the following. First, we consider the abstract equation \eqref{abstreq3} in an (increasing) scale of Banach spaces. The assumptions on $A$ and $Z$  will be abstract and slightly generalized versions of those in \cite{FKKozK2014}. For this equation we prove the existence result (Theorem~\ref{exist}) and show the uniqueness of the `integral curves' of the differential equation (Theorem~\ref{unique}). As was mentioned above, for any $t\in(0,T)$, the solution to \eqref{abstreq3} belongs to a space of the scale, more precisely, $u(t)\in B_\a$, for all $\a>\a_t$, for some `minimal' $\a_t$. In general, one can not check whether $u(t)\in B_{\a_t}$. Therefore, it is quite natural to consider \eqref{abstreq3} in the scale of projective limits $\bigl\{\bigcap_{\beta>\a} B_\beta\bigr\}_{\a\in I}$, $I\subset\R$ (Proposition~\ref{localisation}). This allows to prove rigorously that the flow $t\mapsto u(t)$ is indeed a unique and satisfies the semigroup property; surely, on a finite time interval only (Proposition~\ref{propag}). The problem here was that, in contrast to the classical existence and uniqueness Picard--Lindel\"{o}f theorem for ordinary differential equations, where the finite time interval was defined by the estimate on the right hand side of the equation, in Ovsyannikov's methods we should {\it choose} the initial and the terminal spaces, and they generate the time interval. As a result, we need to be sure that by choosing more wider terminal space we will have the same solution as in the smaller terminal space (with the same initial space), see also Remark~\ref{optimaltime} below. Also one should verify that, by choosing the properly small intermediate moment of time, we could start a new initial value problem at that time and get the same solution thereafter as if we would consider the first problem on a bigger time interval.

The second aim is to consider the convergence of the solutions to the abstract equations \eqref{abstreq3-eps}. For the case $A_\eps=0$, $\eps>0$, the corresponding abstract statement was proved in \cite{FKO2011a}. In Theorem~\ref{thm:conv}, we generalize this scheme by a simple modification of the statements presented in \cite{FKKozK2014}, for the particular model discussed there.

Finally, in Section~\ref{sec:appl}, we consider birth-and-death dynamics using the technique above. The dynamics are a `combination' of those studied in \cite{FK2009} and \cite{FKKZ2014}; it describes the evolution in the course of which the elements disappear (die) more intensively in that regions of the space, where the amounts of their `neighbours' are too big or too small. The birth rate of the dynamics is assumed to be constant in the space: the elements appear from the outside `reservoir' of the system. We prove that the corresponding dynamics exist on a finite time interval and realize the Vlasov-type scaling. This yields a nonlocal nonlinear kinetic equation of a new type. In particular, this equation may have one or three positive stationary points depending on the values of parameters of the system, see Remark~\ref{statpoints} below. The detailed  analysis of this equation will be done in a sequel paper.

The author is thankful to Prof. Dr. Yuri Kondratiev, Prof. Dr. Yuri Kozitsky, and Dr. Oleksandr Kutoviy for fruitful discussions.

\section{Main results}

Let us fix several arrangements. We will use notations like $(B,\lVert\cdot\rVert)$ to say that $B$ is a Banach space with a norm $\lVert\cdot\rVert$. Let $A$ be an (unbounded) linear operator on $B$ with a domain $D\subset B$, we will denote this by $(A,D)$. Note that here and below any inclusion allows equality.
A strongly continuous semigroup of linear bounded operators $S(t)$, $t\geq0$, on $B$ will be called a $C_0$-semigroup. By e.g. \cite[Proposition~I.5.5]{EN2000}, there exist $\omega\in\R$ and $\nu\geq1$ such that $\lVert S(t)\rVert\leq \nu e^{\omega t}$, $t\geq0$, i.e. $S(t)$ is exponentially bounded. If $C$ is a linear subset of $B$, we will denote the closure of $C$ with respect to the norm of $B$ by $\overline{C}$. Note that if $C$ is closed, i.e. $\overline{C}=C$, then $(C,\lVert\cdot\rVert_B)$ is a Banach space as well.

Our first assumption is about a scale of Banach spaces where the dynamics will exist.
\begin{assume}\label{as:scale}
Let $\al>0$, $\au\in(0,\infty]$ be fixed; set $I=(\al,\au)$. Let $\mathbf{B}:=(B_\a)_{\a\in I}$ be a family of Banach spaces $(B_\a,\lVert\cdot\rVert_\a)$ which is supposed to be increasing, i.e.
\begin{equation}\label{incrfamily}
  B_{\a'}\subset B_{\a''}, \qquad \lVert\cdot\rVert_{\a'}\geq\lVert\cdot\rVert_{\a''}, \qquad \a'\leq\a'', \quad \a',\a''\in I.
\end{equation}
Suppose, additionally, that, $u\in B_{\a'}\subset B_{\a''}$, $u=0$ in $B_{\a''}$ yields $u=0$ in $B_{\a'}$, for any $\a',\a''$ as above.
\end{assume}

Next, we consider a linear mapping $A$ on $B_I:=\bigcup_{\a\in I} B_\a$ which acts in each of spaces of the scale $\mathbf{B}$ and satisfies the following assumption.
\begin{assume}\label{as:semgen}
Let, for any $\a\in I$, in the space $B_\a$, there exist a closed linear subset $C_\a$ and its dense linear subset $D_\a$, i.e. $D_\a\subset C_\a\subset B_\a$ and $\overline{D_\a}=C_\a=\overline{C_\a}$. Let $A:B_I\to B_I$ be linear operators and such that, for any $\a\in I$, the operator $(A,D_\a)$ is a generator of a $C_0$-semigroup $S_\a(t)$ on the Banach space $(C_\a,\lVert\cdot\rVert_\a)$. Assume that, for any $\a'\in I$ with $\a'<\a$, $B_{\a'}\subset D_\a$; $B_{\a'}$ is $S_\a(t)$-invariant, and $S_\a(t)\upharpoonright_{B_{\a'}}=S_{\a'}(t)$. Suppose also that the constants $\nu\geq1$, $\omega\in\R$, in the definition of quasi-boundedness for $S_\a(t)$ are independent in $\a$, i.e.
\begin{equation}\label{uniqbdd}
  \lVert S_\a(t)\rVert\leq \nu e^{\omega t}, \qquad t\geq0, \a\in I.
\end{equation}
\end{assume}

Finally, we will deal with a linear mapping $Z$ on $B_I$ which may be considered as a bounded operator between each two spaces of the scale $\mathbf{B}$.
\begin{assume}\label{as:Ovsgen}
Let $M,N: I\to (0,\infty)$ be increasing continuous functions. Let, for any $\a^*\in I$ and for any $\a',\a''\in (\al,\a^*]$ with $\a'<\a''$, $Z$ be a bounded linear operator from $B_{\a'}$ to $B_{\a''}$, such that the following estimate holds:
\begin{equation}\label{singnorm}
\lVert Z u \rVert_{\a''} \leq \Bigl( \frac{M(\a^*)}{\a''-\a'}+N(\a^*)\Bigr) \lVert u\rVert_{\a'}, \quad u\in\ B_{\a'}.
\end{equation}
(Note that $M,N$ may depend on $\al$.)
\end{assume}

Under Assumptions above, consider the following function
\begin{equation}\label{timeint}
T(\a,\beta):=\dfrac{\beta-\a}{e \nu M(\beta)}, \qquad \beta\geq\a>\al.
\end{equation}

\begin{theorem}\label{exist}
Let Assumptions \ref{as:scale}--\ref{as:Ovsgen} hold. Let $\a^*\in I$ and $s\geq0$. Take an arbitrary $\a_s\in (\al,\a^*)$ and set $T:=T(\a_s,\a^*)$. Then, for any $u_s\in B_{\a_s}$, there exists a function $u:[s,s+T)\to B_{\a^*}$ such that
\begin{enumerate}
\item $u$ is continuous on $[s,s+T)$ and continuously differentiable on $(s,s+T)$;
\item for any $t\in(s,s+T)$, $Au(t)\in B_{\a^*}$ and $Zu(t)\in B_{\a^*}$;
\item $u$ solves the following differential equation:
\begin{equation}\label{IVP}
\dfrac{d}{dt} u(t)=Au(t)+Zu(t), \quad t\in(s,s+T),
\end{equation}
\item $u(s)=u_s$.
\end{enumerate}
\end{theorem}
\begin{proof}
We will follow the scheme from \cite{FKKozK2014}.
Take arbitrary $\Tau\in(0,T)$. By the continuity of $M$, there exists $\a\in(\a_s,\a^*)$ such that $\Tau <T(\a_s,\a)=:T'$. Let $q=q(\Tau,T,T')>1$ be such that $q\Tau<\min\{T,T'\}$. For any $n\in\N$, consider the following partition of the interval $[\a_s,\a]$ on $2n+2$ parts:
\begin{equation}\label{partition}
  \a^{(2j,n)}=\a_s+j(\delta_1+\delta_2), \quad  \a^{(2j+1,n)}=\a^{(2j,n)}+\delta_1,
\end{equation}
where $j=0,\ldots,n$ and
\begin{equation}\label{deltas}
  \delta_1=\frac{(q-1)(\a-\a_s)}{q(n+1)}, \qquad \delta_2=\frac{\a-\a_s}{qn}.
\end{equation}
In particular, $\a^{(0,n)}=\a_s$, $\a^{(2n+1,n)}=\a$, and
 \begin{equation}\label{dopest1}
   \a^{(2j+2,n)}-\a^{(2j+1,n)}=\delta_2, \quad j=0,\ldots,n-1.
\end{equation}
For an $n\in\N$, consider  a mapping on $B_I$
\begin{equation}\label{kernelU}
  U_\a^{(n)}(t,t_1,\ldots,t_n)=S_{\a}(t-t_1)ZS_{\a}(t_1-t_2)Z\ldots S_{\a}(t_{n-1}-t_n)ZS_{\a}(t_n),
\end{equation}
where we set $t_0:=t$.
Then, for any $\tau>0$ and $u\in B_{\a^{(2j,n)}}\subset D_{\a^{(2j+1,n)}}\subset C_{\a^{(2j+1,n)}}$, one has $S_{\a}(\tau)u=S_{\a^{(2j,n)}}(\tau)u\in B_{\a^{(2j+1,n)}}$ and hence $ZS_{\a}(\tau)u\in B_{\a^{(2j+2,n)}}$, $j=0,\ldots,n-1$, with
\[
\lVert ZS_{\a}(\tau)u \rVert_{\a^{(2j+2,n)}}\leq \nu e^{\omega \tau}
\Bigl( \frac{M(\a)}{\delta_2}+N(\a)\Bigr) \lVert u\rVert_{\a^{(2j,n)}}
\]
As a result, $U_\a^{(n)}(t,t_1,\ldots,t_n)u_s\in B_{\a}$ with
\begin{align}
  \lVert U_\a^{(n)}(t,t_1,\ldots,t_n)u_s \rVert_{\a}&\leq \nu^{n+1} e^{\omega t}\Bigl( \frac{M(\a)}{\delta_2}+N(\a)\Bigr)^n \lVert u\rVert_{\a_s}\notag\\
  &= \nu e^{\omega t}\Bigl( \frac{qn}{e T'}+\nu N(\a)\Bigr)^n \lVert u\rVert_{\a_s}. \label{estforUn}
\end{align}

Set now $V_\a^{(1)}(s,t)u_s:=\int_s^t U^{(1)}(t,t_1)u_s\,dt_1$  and
\[
V_\a^{(n)}(s,t)u_s:=\int_s^t\int_s^{t_1}\ldots\int_s^{t_{n-1}} U_\a^{(n)}(t,t_1,\ldots,t_n)u_s\,dt_{n}\ldots dt_1.
\]
Therefore, the series
\begin{equation}\label{series}
  S_{\a}(t-s) u_s +\sum_{n=1}^\infty V_\a^{(n)}(s,t)u_s
\end{equation}
is majorized in $B_{\a}$ by the series
\begin{equation}\label{majser}
  \nu e^{\omega t}\lVert u_s\rVert_{\a_s} \sum_{n=0}^\infty \frac{1}{n!}\Bigl( \frac{qn}{e T'}+\nu N(\a)\Bigr)^n (t-s)^n,
\end{equation}
which converges uniformly on $t\in[s,s+\Tau]$, as
\begin{multline}
\biggl(\frac{1}{n!}\Bigl( \frac{qn}{e T'}+\nu N(\a)\Bigr)^n (t-s)^n\biggr)^{\frac{1}{n}}\\ \sim \frac{e}{n}\Bigl( \frac{qn}{e T'}+\nu N(\a)\Bigr)(t-s)\sim \frac{q(t-s)}{T'}\leq\frac{q\Tau}{T'}<1.\label{sim}
\end{multline}
Therefore, the series \eqref{series} converges uniformly on $t\in[s,s+\Tau]$ in $B_{\a}$  to a function $u(t)\in B_{\a}$. Evidently, each term of \eqref{series} is continuous as a mapping $[s,s+\Tau]\to B_{\a}$ thus $u(t)$ is continuous as well. Since the norm in $B_\a$ is stronger  than in $B_{\a^*}$ one has that $[s,s+\Tau]\ni t\mapsto u(t)\to B_{\a^*}$ is also continuous.

Consider now the series of derivatives of the terms from \eqref{series}. Each of them belongs to $B_\a\subset D_{\a^*}$, thus
\begin{align}
  \frac{d}{dt}V^{(n)}_\a(s,t)
&=\int_s^{t} \int_s^{t_2} \ldots \int_s^{t_{n-1}} U^{(n)}(t,t,t_2\ldots,t_n)u_s\,dt_{n}\ldots dt_1\notag\\&\quad+\int_s^{t} \int_s^{t_1} \ldots \int_s^{t_{n-1}} AU_\a^{(n)}(t,t_1,\ldots,t_n)u_s\,dt_{n}\ldots dt_1 \notag\\
&=Z V^{(n-1)}_\a(s,t)+AV^{(n)}_\a(s,t)\label{sda}
\end{align}
is well-defined and belongs to $B_{\a^*}$. By the same arguments as above the series of derivatives converges uniformly on $t\in[s,s+\Tau]$ in $B_{\a^*}$ and hence its sum is equal to $\frac{d}{dt}u(t)$. Note also that $u(t)\in B_\a\subset D_{\a^*} $ and hence $Au, Zu\in B_{\a^*}$. Thus, by \eqref{sda}, $u(t)$ satisfies \eqref{IVP}. Since $\Tau\in (0,T)$ was arbitrary, the statement is proved.
\end{proof}

\begin{remark}
  It is easy to see that the summand $N(\a^*)$ in \eqref{singnorm} might be changed on $\dfrac{N(\a^*)}{(\a''-\a')^\delta}$, for an arbitrary $\delta\in(0,1)$, without any changes in \eqref{timeint}.
\end{remark}

\begin{corollary}\label{cor:est}
Let conditions and notations of Theorem~\ref{exist} hold. Set
\[
N_*:=\sup_{\a\in[\a_s,\a^*]} N(\a)<\infty, \qquad
T_*:=\sup_{\a\in[\a_s,\a^*]} T(\a_s,a)<\infty.
\]
Then, for any $\Tau\in(0,T)$ and $\a\in(\a_s,\a^*)$ such that $\Tau<T(\a_s,\a)=:T'$, and for any $q\in \bigl(1,\frac{T'}{\Tau}\bigr)$,
\begin{equation}\label{est}
  \lVert u(t)\rVert_\a\leq \frac{C }{T'-q\Tau}e^{\omega t} \lVert u_s\rVert_{\a_s}, \qquad t\in[s,s+\Tau],
\end{equation}
  where $C=C(\nu,T_*,N_*)>0$.
\end{corollary}
\begin{proof}
  By \eqref{majser}, we have, for $t\in[s,s+\Tau]$,
  \begin{align*}
    \lVert u(t)\rVert_\a & \leq \nu e^{\omega t}\lVert u_s\rVert_{\a_s} \sum_{n=0}^\infty \frac{1}{n!}\Bigl(\frac{n}{e }\Bigr)^n \Bigl(\frac{q\Tau}{T'}\Bigr)^n\Bigl( 1+\frac{e \nu T' N(\a)}{qn}\Bigr)^n , \\
    & \leq \nu e^{\omega t}\lVert u_s\rVert_{\a_s} \sum_{n=0}^\infty \frac{1}{e}\Bigl(\frac{q\Tau}{T'}\Bigr)^n \Bigl( 1+\frac{e \nu T_* N_*}{n}\Bigr)^n\\&
    \leq \nu e^{\omega t+e \nu T_* N_*-1}\lVert u_s\rVert_{\a_s} \frac{T_*}{T'-q \Tau}
  \end{align*}
  where we have used the following elementary inequalities
  \[
  n!\geq e \Bigl(\frac{n}{e}\Bigr)^n, \qquad \Bigl(1+\frac{x}{n}\Bigr)^n\leq e^x, \qquad n\in\N, x>0.\qedhere
  \]
\end{proof}
Now we are ready to formulate a uniqueness result.
\begin{theorem}\label{unique}
Let Assumptions \ref{as:scale}--\ref{as:Ovsgen} hold. Let $\a^*\in I$ and $T>0$. Let, for some $s\geq0$ and $\tau\in[s,s+T)$ continuous functions $[\tau,s+T)\to u_i(t)\in B_{\a^*}$, $i=1,2$ satisfy to the differential equation \eqref{IVP} on $(\tau,s+T)$ in $B_{\a^*}$. Suppose that there exists $\a_\tau\in (\al,\a^*)$ such that $\tau+T(\a_\tau,\a^*)\geq s+T$ and $u_1(\tau)=u_2(\tau)=:u_\tau \in B_{\a_\tau}$. Then $u_1(t)=u_2(t)$ in $B_{\a^*}$, for any $t\in(\tau,s+T)$.

In particular, the function $u$ in Theorem~\ref{exist} is unique.
\end{theorem}
\begin{proof}
Take an arbitrary $\Tau\in(\tau-s,T)$ thus $\tau< s+\Tau<s+T<\tau+T(\a_\tau,\a^*)$. Since $\a^*\in I$, $I$ is an open interval and $M$ is continuous on $I$, there exists $\a^\circ\in I$ such that $\a^\circ>\a^*$ and $s+\Tau< \tau +T(\a_\tau,\a^\circ)$. Let $u(t):=u_1(t)-u_2(t)$, $t\in[\tau,s+T)$. Then $u$ solves \eqref{IVP} on $(\tau,s+T)$ with $u(\tau)=0\in B_{\a_\tau}$. It is enough to prove that $u(t)=0\in B_{\a^\circ}$ (and thus $u(t)=0\in B_{\a^*}$). Since the norm in $B_{\a^*}$ is stronger than the norm in $B_{\a^\circ}$, $u(t)$ solves \eqref{IVP} in $B_{\a^\circ}$ as well. Then, one has the following equality in $B_{\a^\circ}$
\begin{equation}\label{das}
  u(t)=\int_\tau^t S_{\a^\circ}(t-t') Z u(t') dt', \qquad t\in[\tau,s+\Tau].
\end{equation}
However, $u(t)\in B_{\a^*}$ hence, for any $\a'\in(\a^*,\a^\circ)$, one can take any $\a''\in(\a^*,\a')$ and consider the right hand side of \eqref{das} as follows: $u(t')\in B_{\a^*}$, $Zu(t')\in B_{\a''}\subset D_{\a'}$, $S_{\a^\circ}(t-t') Z u(\tau) =
S_{\a'}(t-t') Z u(\tau) \in B_{\a'}\subset B_{\a^\circ}$, and all the mappings are continuous.
Therefore, one can iterate \eqref{das} $n$ times and consider partition \eqref{partition}--\eqref{deltas} with $\a^\circ,\a^*$ in place of $\a^*,\a_s$ respectively. As a result, one gets, cf.~\eqref{estforUn},
\begin{equation}\label{estuniq}
  \lVert u(t)\rVert_{\a^\circ}\leq e^{\omega t}\Bigl( \frac{q'n}{e T(\a^*,\a^\circ)}+\nu N(\a^\circ)\Bigr)^n \frac{(t-\tau)^n}{n!}\lVert u(t)\rVert_{\a^*}, \quad t\in[\tau,s+\Tau],
\end{equation}
with a properly chosen $q'>1$.  Let now $N\in\N$ be big enough to guarantee that \eqref{estuniq} implies $\lVert u(t)\rVert_{\a^\circ}=0$, for $t\in[\tau,\tau+\sigma]$, $\sigma:=\frac{s+\Tau-\tau}{N}$, i.e. $N>\frac{s+\Tau-\tau}{T(\a^*,\a^\circ)}>0$. Thus $u(t)=0$ in $B_{\a^\circ}$ and hence in $B_{\a^*}$, for $t\in[\tau,\tau+\sigma]$. Repeat now the same arguments with initial zero value at $t=\tau+\sigma$, it will lead to the zero solution in $B_{\a^*}$ on $[\tau+\sigma,\tau+2\sigma]$ and so on. As a result, we will get that $u(t)=0$ in $B_{\a^*}$ on the whole $[\tau,s+\Tau]$, and since $\Tau$ was arbitrary we will have the uniqueness on the $[\tau,s+T)$.
\end{proof}

\begin{remark}\label{optimaltime}
  In applications, we often have an estimate like \eqref{singnorm} with $\tilde{M}(\a'',\a')$ in place of $\M(\a^*)$, with a function $\tilde{M}$ which is increasing in the first variable and decreasing in the second one thus $\tilde{M}(\a'',\a')\leq\tilde{M}(\a^*,\al)=:M(\a^*)$, and $M$ is an increasing function. cf. \cite{FKKoz2011,BKKK2011,FKKO2014,FKKZ2014}. Note also, that the function $T(\a,\beta)$ is not typically increasing in $\beta$, see \eqref{timeint} and the references above; therefore, the bigger terminal space $B_{\a^*}$ does not necessarily lead to a wider time interval. Note that, in the cited references, the function $T(\a,\cdot)$ had a unique maximum point.
\end{remark}

For any $\a\in I$, consider the set
\begin{equation}\label{proj}
  B_{\a+}:=\bigcap_{\beta>\a} B_\beta,
\end{equation}
which may be endowed by the sequential topology of a projective space, see e.g.~\cite{Ber1986}, i.e. $u_n\to u$ in $B_{\a+}$ if and only if $u_n\to u$ in all $B_\beta$, $\beta>\a$ (of course, by \eqref{incrfamily}, it is sufficient to take $\beta\in (\a,\a+\delta)$ only, for some $\delta>0$). Stress that, under Assumption~\ref{as:Ovsgen}, $Z$ is continuous on $B_{\a+}$.

\begin{proposition}\label{localisation}
In conditions and notations of Theorem~\ref{exist},
\begin{enumerate}
\item the mapping $B_{\a_s}\ni u_s\mapsto u(t)\in B_{\a^*}$ is continuous, uniformly in $t\in[s,s+\Tau]\subset[s,s+T)$;
\item for any $t\in(s,s+T)$, there exist
\begin{equation}\label{localis}
\a(t,s,\a_s):=\inf\bigl\{\a\in[\a_s,\a^*)\bigm| u(t)\in B_{\a}\bigr\}<\a^*.
\end{equation}
 such that $u(t)\in B_{\a(t,s,\a_s)+}$ and the mapping $B_{\a_s}\ni u_s\mapsto u(t)\in B_{\a(t,s,\a_s)+}$ is also continuous, uniformly in $t\in[s,s+\Tau]\subset[s,s+T)$;
\item one can take $u_s\in B_{\a_s+}$; then all previous statements remain true with $B_{\a_s+}$ in place of $B_{\a_s}$ only.
\end{enumerate}
\end{proposition}

\begin{proof} We will use details of the proof of Theorem~\ref{exist}.

 (1) Let $v(t)$ solve \eqref{IVP} on $[s,s+T)$ with $v(0)=v_s\in B_{\a_s}$. Then we will have that, for any $t\in[s,s+\Tau]\subset[s,s+T)$, and for the same $\a\in(\a_s,a^*)$, $q>1$,
 \begin{multline}\label{sadf}
  \lVert u(t)-v(t)\rVert_{\a^*}\leq \lVert u(t)-v(t)\rVert_{\a} \\\leq \max\bigl\{e^{\omega(s+ \Tau)},1\bigr\}\lVert u_s-v_s\rVert_{\a_s} \sum_{n=0}^\infty \frac{1}{n!}\Bigl( \frac{qn}{e T'}+\nu N(\a)\Bigr)^n \Tau^n,
 \end{multline}
  that implies the needed continuity, as, recall, $\a$ depends on $\Tau$, $q=q(T,T')=q(\a,\a_s)$ and thus the estimate is uniform in $t\in[s,s+\Tau]$.

   (2) Recall that the solution $u(t)$ in $B_{\a^*}$ to \eqref{IVP} is given on $[s,s+T)$ by \eqref{series} and, for a chosen $t\in[s,s+T)$, the value $u(t)$, as a matter of fact, belongs to $B_\a$, for any $\a\in[\a_s,\a^*)$ such that $s<t<s+T(\a_s,\a)$. Since $T(\a_s,\a_s)=0$ we have by the continuity arguments that there exists $\a^\circ=\a(t,s,\a_s)$ such that $T(\a_s,\a^\circ)=t-s$ and there exists an open subinterval of $(\a_s,\a^*)$ where $T(\a_s,\a)>t-s$ (or a union of such subintervals). Thus, the set in \eqref{localis} is non-empty, the infimum does exist, that yields the first statement. Next, by \eqref{sadf}, the mapping $B_{\a_s}\ni u_s\mapsto u(t)\in B_\a$ will be continuous, for any $\a>\a(t,s,\a_s)$, uniformly in $t\in[s,s+\Tau]$. This fulfilled the statement.

   (3) Let $u_s\in B_{\a_s+}$. For any $\Tau\in(0,T)$, one can choose $\a_s'\in(\a_s,\a^*)$ with $\a_s'-\a_s$ small enough to guarantee that $\Tau<T(\a_s',\a^*)<T$. Then one can repeat all arguments above and get that there exists solution to \eqref{IVP} on $[s,s+\Tau]$ in $B_{\a^*}$ with $u(0)=u_s\in B_{\a_s'}$ such that $B_{\a_s'}\ni u_s\mapsto u(t)\in B_{\a^*}$ (and $B_{\a^*}$ may be replaced on $B_{\a(t,s,\a_s')+})$. If we take now $\a_s''\in(\a_s,\a^*)$, $\a''\neq\a'$, then the solution will be given by {\it the same} series \eqref{series} which converges {\it in the same} space $B_{\a^*}$ (or even in a smaller space). The difference will be in the denominator of \eqref{est} only: namely, $\|u(t)\|_{\beta}\leq \bar C(\beta) \|u_s\|_\gamma$, $\gamma>\a_s$, $\beta\in (\a(t,s,\a_s'),\a^*)$, $t\in[s,s+\Tau]$. This, naturally, implies the continuity of the mapping $B_{\a_s+}\ni u_s\mapsto u(t)\in B_{\a(t,s,\a_s)+}$.
\end{proof}

\begin{remark}
  Note that, by the comparison series criterium, we have from \eqref{sim} that the majorized series for \eqref{series} diverges for $\a=\a(t,s,\a_s)$. However, one can not state that $u(t)\notin B_{\a(t,s,\a_s)}$.
\end{remark}

According to Theorem~\ref{localisation}, for any $s\geq0$, $\a^*\in I$, $\a_s\in(\al,\a^*)$, $t\in[s,s+T(\a_s,\a^*))$ one can define the mapping $U(s,t): B_{\a_s+}\to B_{\a(t,s,\a_s)+}$ given by
\begin{equation}\label{propag}
U(s,t)u_s=u(t).
\end{equation}
\begin{proposition}\label{sgrproperty}
Let $s\geq0$, $\a^*\in I$, $\a_s\in B_{\a_s}$ be arbitrary. Let $t>\tau>s$ be such that $\tau<s+T(\a_s,\a^*)$, $t<\min\big\{\tau+T(\a(\tau,s,\a_s),\a^*),s+T(\a_s,\a^*)\bigr\}$. Then
\begin{equation}\label{evol}
U(s,t)u_s=U(\tau,t)U(s,\tau)u_s.
\end{equation}
\end{proposition}
\begin{proof}
  The statements follows from the uniqueness Theorem~\ref{unique}. Indeed, $\sigma:=\min\big\{\tau+T(\a(\tau,s,\a_s),\a^*),s+T(\a_s,\a^*)\bigr\}> s$, therefore, $T:=\sigma-s>0$. Then, by the construction of the mapping \eqref{propag}, both functions $U(s,t)u_s$  and $U(\tau,t)U(s,\tau)u_s$ solves \eqref{IVP} on $(\tau,s+T)$ and they are both equal to $U(s,\tau)u_s$ at $t=\tau$. Hence, by Theorem~\ref{unique} they coincide on $(\tau,\tau+T)$ as well.
\end{proof}

\begin{remark}
In the same manner as before, one prove the following statement. In conditions and notations of Theorem~\ref{exist}, suppose, additionally that there exists $\a^{**}\in I$ such that $\a^*<\a^{**}$ and Assumption~\ref{as:Ovsgen} holds for $ \a^{**}$ in place of $\a^*$. Set $\tilde{T}:=T(\a_s,\a^{**})$ and $T_0:=\min\{T,\tilde{T}\}$. Let $\tilde{u}:[s,s+\tilde{T})\to B_{\a^{**}}$ be the solution to \eqref{IVP} according to Theorem~\ref{exist}. Then, for any $t\in[s,s+T_0)$, $\tilde{u}(t)=u(t)\in B_{\a^*}$.
\end{remark}

\begin{assume}\label{as:semgen-eps}
Let $\{ D_\a, C_\a\}_{\a\in I}$ be such as in Assumption~\ref{as:semgen}. Let $A_\eps:B_I\to B_I$, $\eps\geq0$ be linear operators, such that, for any $\eps\geq0$ and for any $\a\in I$, the operator $(A_\eps,D_\a)$ is a generator of a $C_0$-semigroup $S_{\a,\eps}(t)$ on the Banach space $(C_\a,\lVert\cdot\rVert_\a)$. Assume that, for any $\a'\in I$ with $\a'<\a$, $B_{\a'}\subset D_\a$; $B_{\a'}$ is $S_{\a,\eps}(t)$-invariant, and $S_{\a,\eps}(t)\upharpoonright_{B_{\a'}}=S_{\a',\eps}(t)$. Suppose also that the constants $\nu\geq1$, $\omega\in\R$ are such that
\begin{equation}\label{uniqbdd-eps}
  \lVert S_{\a,\eps}(t)\rVert\leq \nu e^{\omega t}, \qquad t\geq0, \a\in I,\eps\geq0.
\end{equation}
\end{assume}

\begin{assume}\label{as:Ovsgen-eps}
Let $M,N: I\to (0,\infty)$ be increasing continuous functions. Let, for any $\a^*\in I$ and for any $\a',\a''\in (\al,\a^*]$ with $\a'<\a''$, $Z_\eps$, $\eps\geq0$, be bounded linear operators from $B_{\a'}$ to $B_{\a''}$, such that the following estimate holds:
\begin{equation}\label{singnorm-eps}
\lVert Z_\eps u \rVert_{\a''} \leq \Bigl( \frac{M(\a^*)}{\a''-\a'}+N(\a^*)\Bigr) \lVert u\rVert_{\a'}, \quad u\in\ B_{\a'}.
\end{equation}
\end{assume}

\begin{theorem}\label{thm:conv}
  Let Assumption~\ref{as:scale},~\ref{as:semgen-eps},~\ref{as:Ovsgen-eps} hold.
  Let $P_\eps, p_\eps:I\to (0,\infty)$, $\eps>0$ be increasing continuous functions, such that
  \begin{equation}\label{seq0}
   \lim_{\eps\to0}p_\eps(\a)=\lim_{\eps\to0}P_\eps(\a)=0, \quad \a\in I,
   \end{equation}
  and let $r\in\N$.
  Let $\a^*\in I$ and $\a',\a''\in (\al,\a^*)$, $\a'<\a''$ be arbitrary, and suppose that
  \begin{equation}\label{convsg}
  \lVert S_{\a'',\eps}(t)u -S_{\a'',0}(t) u\rVert_{\a''} \leq p_\eps(\a^*)
    e^{\omega t} \lVert u\rVert_{\a'}, \qquad t\in(s,s+T), \eps>0, u\in C_{\a'},
  \end{equation}
  and
  \begin{equation}\label{convZ}
  \lVert Z_\eps u -Z_0 u\rVert_{\a''} \leq
   \sum_{j=1}^r \frac{P_\eps(\a^*)}{(\a''-\a')^j}\lVert u\rVert_{\a'}, \qquad \eps>0, u\in B_{\a'}.
  \end{equation}
  Let $s\geq0$, $\a_s\in(\al,\a^*)$, $u_{s,\eps}, u_{s,0}\in B_{\a_s}$ be arbitrary, and suppose that
  \begin{equation}\label{initconv}
    \lim_{\eps\to0}\lVert u_{s,\eps} - u_{s,0}\rVert_{\a_s}=0.
  \end{equation}
  Then, for any $\eps\geq0$, there exist a unique solution to the differential equation
  \begin{equation}\label{IVPeps}
    \begin{cases}
      \dfrac{d}{dt} u_\eps(t) = (A_\eps+Z_\eps) u_\eps(t),\qquad t\in(s,s+T)\\[2mm]
      u_\eps(s)= u_{s,\eps},
    \end{cases}
  \end{equation}
  in $B_{\a^*}$, where $T=T(\a_s,\a^*)$; and, moreover, for any $\Tau\in(s,s+T)$,
  \begin{equation}\label{conv}
     \lim_{\eps\to0}\sup_{t\in[s,s+\Tau]}\lVert u_{\eps} (t)- u_{0}(t)\rVert_{\a^*}=0.
  \end{equation}
\end{theorem}

\begin{proof}
  The existence and uniqueness of solutions to \eqref{IVPeps} follow directly from Theorems~\ref{exist} and~\ref{unique}. By the proof of Theorem~\ref{exist}, it is easy to see that there exists $\a=\a(\Tau)\in(\a_s,\a^*)$, which does not depend on $\eps$, such that
    \begin{equation}\label{series-eps}
      u_\eps(t)=\sum_{n=0}^\infty V_{\a,\eps}^{(n)}(s,t)u_{s,\eps}, \qquad\eps\geq0,
    \end{equation}
    where $V_{\a,\eps}^{(0)}(s,t):=S_{\a,\eps}(t-s)$ and
    \begin{align*}
    V_{\a,\eps}^{(n)}(s,t)&:=\int_s^t\int_s^{t_1}\ldots\int_s^{t_{n-1}} U_{\a,\eps}^{(n)}(t,t_1,\ldots,t_n)\,dt_{n}\ldots dt_1,\\
    U_{\a,\eps}^{(n)}(t,t_1,\ldots,t_n)&:=S_{\a,\eps}(t-t_1)Z_\eps S_{\a,\eps}(t_1-t_2)Z_\eps\ldots S_{\a,\eps}(t_{n-1}-t_n)Z_\eps S_{\a,\eps}(t_n),
    \end{align*}
    and the series \eqref{series-eps} converges in $B_\a$. Recall that $\Tau<T'=T(\a_s,\a)$ and let, as before, $q\in\bigl(1,\frac{T'}{\Tau}\bigr)$.

    Therefore, by the proof of Corollary \ref{cor:est},
    \begin{align*}
      &\quad \lVert u_\eps(t)-u_0(t)\rVert_\a  \\&\leq \sum_{n=0}^\infty \bigl\lVert \bigl(V_{\a,\eps}^{(n)}(s,t)-V_{\a,0}^{(n)}(s,t)\bigr)u_{s,0}\bigr\rVert_\a + \sum_{n=0}^\infty \bigl\lVert V_{\a,0}^{(n)}(s,t)(u_{s,\eps}-u_{s,0})\bigr\rVert_\a\notag\\
      &\leq \bigl\lVert \bigl(S_{\a,\eps}(t-s)-S_{\a,0}(t-s)\bigr)u_{s,0}\bigr\rVert_\a
      \notag\\&\quad +\sum_{n=1}^\infty \int_s^t\int_s^{t_1}\ldots\int_s^{t_{n-1}} \bigl\lVert \bigl(U_{\a,\eps}^{(n)}(t,t_1,\ldots,t_n)-U_{\a,0}^{(n)}(t,t_1,\ldots,t_n)\bigr)u_{s,0}\bigr\rVert_\a\,dt_{n}\ldots dt_1 \notag\\&\quad+\frac{Ce^{\omega (s+\Tau)}}{T'-q\Tau}\lVert u_{s,\eps}-u_{s,0}\rVert_{\a_s}.\notag
    \end{align*}
    Denote, for simplicity of notations,
    $Q_{\a,\eps}(t):=S_{\a,\eps}(t)-S_{\a,0}(t), \quad t\geq0, \qquad R_{\eps}:=Z_\eps-Z_0$.
    Then, for $n\geq1$,
    \begin{align*}
      \quad &U_{\a,\eps}^{(n)}(t,t_1,\ldots,t_n)-U_{\a,0}^{(n)}(t,t_1,\ldots,t_n) \\& =Q_{\a,\eps}(t-t_1) Z_\eps S_{\a,\eps}(t_1-t_2)Z_\eps\ldots S_{\a,\eps}(t_{n-1}-t_n)Z_\eps S_{\a,\eps}(t_n)\\
       &\quad+ S_{\a,0}(t-t_1)R_\eps S_{\a,\eps}(t_1-t_2)Z_\eps\ldots S_{\a,\eps}(t_{n-1}-t_n)Z_\eps S_{\a,\eps}(t_n)\\
       &\quad+ S_{\a,0}(t-t_1)Z_0 Q_{\a,\eps}(t_1-t_2)Z_\eps\ldots S_{\a,\eps}(t_{n-1}-t_n)Z_\eps S_{\a,\eps}(t_n)\\
       &\quad+ \ldots\\
       &\quad+ S_{\a,0}(t-t_1)Z_0 S_{\a,0}(t_1-t_2)Z_0\ldots Q_{\a,\eps}(t_{n-1}-t_n)Z_\eps S_{\a,\eps}(t_n)\\
       &\quad+ S_{\a,0}(t-t_1)Z_0 S_{\a,0}(t_1-t_2)Z_0\ldots S_{\a,0}(t_{n-1}-t_n)R_\eps S_{\a,\eps}(t_n)\\
       &\quad+ S_{\a,0}(t-t_1)Z_0 S_{\a,0}(t_1-t_2)Z_0\ldots S_{\a,0}(t_{n-1}-t_n)Z_0 Q_{\a,\eps}(t_n).
    \end{align*}
    By using the partition \eqref{partition}--\eqref{deltas}, one gets, cf.~\eqref{estforUn},
    \begin{align*}
      \quad &\lVert U_{\a,\eps}^{(n)}(t,t_1,\ldots,t_n)-U_{\a,0}^{(n)}(t,t_1,\ldots,t_n) \rVert_\a\\& \leq \nu e^{\omega t}\Bigl( \frac{qn}{e T'}+\nu N(\a)\Bigr)^{n-1}
       \biggl(  n p_\eps(\a) + n\nu\sum_{j=1}^r \frac{P_\eps(\a)}{\delta_2^j} \biggr)\lVert u\rVert_{\a_s}\\& \leq \nu e^{\omega t}\Bigl( \frac{qn}{e T'}+\nu N(\a)\Bigr)^{n-1}
       \biggl(  n \nu p_\eps(\a) + n\nu\sum_{j=1}^r \frac{(qn)^j}{(eT')^j}P_\eps(\a)\biggr)\lVert u\rVert_{\a_s}.
    \end{align*}
    As a result,
    \begin{align*}
      &\quad \lVert u_\eps(t)-u_0(t)\rVert_\a  \\
      &\leq p_\eps(\a^*)\nu e^{\omega \Tau}\lVert u_{s,0}\bigr\rVert_{\a,s}
      \\&\quad+P_\eps(\a)\sum_{n=1}^\infty \frac{\Tau^n}{n!}\nu e^{\omega (s+\Tau)}\Bigl( \frac{qn}{e T'}+\nu N(\a)\Bigr)^{n} \biggl(  n\nu\sum_{j=1}^r \frac{(qn)^{j-1}}{(eT')^{j-1}} \biggr)\lVert u_{s,0}\bigr\rVert_{\a,s}
      \notag\\&\quad+p_\eps(\a)\sum_{n=1}^\infty \frac{\Tau^n}{n!}\nu e^{\omega (s+\Tau)}\Bigl( \frac{qn}{e T'}+\nu N(\a)\Bigr)^{n-1}   n\nu\lVert u_{s,0}\bigr\rVert_{\a,s}
      \notag\\&\quad+\frac{Ce^{\omega (s+\Tau)}}{T'-q\Tau}\lVert u_{s,\eps}-u_{s,0}\rVert_{\a_s},\notag
    \end{align*}
    that fulfills the statement, by \eqref{initconv}, \eqref{seq0}; note that convergence of two latter series holds by \eqref{sim}. (As a matter of fact, we have proved the convergence \eqref{conv} with a stronger norm $\lVert \cdot\rVert_\a$.)
\end{proof}

\section{An application to birth-and-death dynamics}\label{sec:appl}

We will start with a brief introduction to the configuration space analysis. More detailed explanation may be found in, e.g., \cite{KK2002,FKK2011a,FKK2014}.

Let $\Bb$ be the set of all bounded Borel subsets of $\X$. The configuration space over space $\X$ consists of all locally
finite subsets (configurations) of $\X$, i.e.
\begin{equation} \label{confspace}
\Ga :=\bigl\{ \ga \subset \X \bigm| |\ga _\La
|<\infty, \ \mathrm{for \ all } \ \La \in {\B}_{\mathrm{b}}
(\X)\bigr\}.
\end{equation}
Here $|\cdot|$ means the cardinality of a~set, and
$\ga_\La:=\ga\cap\La$. The Borel $\sigma$-algebra $\B(\Ga)$ is generated by all mappings $\Ga\ni\ga\mapsto|\ga_\La|\in \N_0:=\N\cup\{0\}$. Let $\Mfm$ be the set of all probability measures $\mu$ on $\bigl(\Ga,\B(\Ga)\bigr)$ such that $\int_\Ga |\ga_\La|^n\,d\mu(\ga)<\infty$, for any $\La\in\Bb$ and $n\in\N$.

Let $\Ga_0$ be the space of all finite configurations from $\X$, i.e.
\begin{equation} \label{confspace0}
\Ga_0 :=\bigl\{ \eta \subset \X \bigm| |\eta |<\infty\bigr\}.
\end{equation}
Then $\Ga_0=\bigsqcup_{n\in\N_0}\Ga^{(n)}$, where $\Ga^{(n)} :=\bigl\{ \eta \subset \X \bigm| |\eta |=n\bigr\}$, $n\in\N_0$. Clearly, $\Ga^{(n)}\sim \widetilde{(\X)^n}/S_n$, where the tilde denotes the product set without diagonals and $S_n$ is the permutation group. This isomorphism provides the natural $\sigma$-algebra $\B(\Ga_0)$ on $\Ga_0$. The Lebesgue--Poisson measure on $\bigl(\Ga_0,\B(\Ga_0)\bigr)$ is defined via the following equality:
\begin{equation}\label{LebPois}
  \int_{\Ga_0} G(\eta) \,d\eta = G^{(0)}+ \sum_{n=1}^\infty \int_{(\X)^n} G^{(n)}(x_1,\ldots,x_n)\,dx_1\ldots dx_n,
\end{equation}
where $G$ is a measurable nonnegative function on $\Ga_0$, which may be identified with the sequence of symmetric functions $G^{(n)}$, namely, $G(\{x_1,\ldots,x_n\})=G^{(n)}(x_1,\ldots,x_n)$, $G(\emptyset)=G^{(0)}\in\R$.

Let $\Bbs$ be the set of all measurable bounded functions $G:\Ga_0\to\R$ such that there exist $N\in\N$ and $\La\in\Bb$ such that, for $n>N$,  $G^{(n)}\equiv0$, and, for $n\leq N$, $G^{(n)}(x_1,\ldots,x_n)=0$ if only $x_i\notin\La$, for some $1\leq i\leq n$. Then, for any $G\in\Bbs$, one can define the following function on $\Ga$:
\begin{equation}
(KG)(\ga ):=\sum_{\eta \Subset \ga }G(\eta ), \quad \ga \in \Ga,\label{K-transform}
\end{equation}
where the summation is
taken over all finite subconfigurations $\eta\in\Ga_0$ of the
(infinite) configuration $\ga\in\Ga$; we denote this by the symbol,
$\eta\Subset\ga $.
The mapping $K$ is linear, positivity preserving,
and invertible, with
\begin{equation}
(K^{-1}F)(\eta ):=\sum_{\xi \subset \eta }(-1)^{|\eta \setminus \xi
|}F(\xi ),\quad \eta \in \Ga_0. \label{k-1transform}
\end{equation}
It can be shown that, for $G\in\Bbs$ with $\La\in\Bb$, $N\in\N$ as above,
$KG(\ga)=KG(\ga_\La)$ and $|KG(\ga)|\leq C (1+|\ga_\La|)^N$, $\ga\in\Ga$. In particular, $KG\in L^1(\Ga,\mu)$, for any $\mu\in\Mfm$.
The correlation function of a measure $\mu\in\Mfm$ is the function $k_\mu:\Ga_0\to\R_+$ which satisfies the identity
\begin{equation}\label{cfdef}
  \int_\Ga (KG)(\ga)\,d\mu(\ga)=\int_{\Ga_0} G(\eta) k_\mu(\eta)\,d\eta,
\end{equation}
for any $0\leq G\in\Bbs$, provided such $k_\mu$ does exist.

We consider a model which is a combination of models discussed in \cite{FK2009} and \cite{FKKZ2014}.

Let $a,\phi:\X\to\R_+:=[0,\infty)$ be measurable nonnegative symmetric functions, i.e. $a(-x)=a(x)$, $\phi(-x)=\phi(x)$, $x\in\X$. Assume that $a,\phi\in L^1(\X,dx)\cap L^\infty(\X,dx)$. Set
\begin{equation}\label{integrof}
  \langle a\rangle:=\int_\X a(x)\,dx, \quad \bar a:=\esssup_{x\in\X} a(x), \quad
  \langle \phi\rangle:=\int_\X \phi(x)\,dx, \quad \bar\phi:=\esssup_{x\in\X} \phi(x).
\end{equation}
Let $m,\la>0$ be constants. For any $F\in K\bigl(\Bbs\bigr)$, we define the mapping
\begin{equation}\label{Mgen}
\begin{split}
  (LF)(\ga)&=\sum_{x\in\ga} \biggl(\sum_{y\in\ga\setminus x} a(x-y)+m\exp\Bigl(-\sum_{y\in\ga\setminus x} \phi(x-y) \Bigr) \biggr) \bigl( F(\ga\setminus x)-F(\ga)\bigr)\\&\quad +\la\int_\X \bigl(F(\ga\cup x)-F(\ga)\bigr)\,dx.
\end{split}
\end{equation}
Here and below we use the notations $\setminus x$ and $\cup x$ instead of more precise $\setminus \{x\}$ and $\cup \{x\}$, respectively. Heuristically, $L$ describes the followig evolution of configurations: during a (small) time $t$ in an arbitrary domain $\La\in\Bb$ a new elements may appear with the probability $\la\, \mathrm{vol}(\La) \, t+o(t)$, whereas the probability for the existing point $x\in\ga$ disappears is equal to
$
\sum_{y\in\ga\setminus x} a(x-y)t+m\exp\bigl(-\sum_{y\in\ga\setminus x} \phi(x-y) \bigr) t+o(t)$.
Thus this probability will be close to $1$ in very dense regions of the space as well as in the almost `uninhabited' places.

Since $F(\ga\cup x)=F((\ga\cup x)_\La)$, the integrant in \eqref{Mgen} equals to $0$ outside of $\La$, thus the integral is well-defined. By the same arguments, the first (outer) sum in \eqref{Mgen} is taken over $x\in\ga_\La$ only. The other sums (in $y$) are, however, infinite. In particular, \eqref{Mgen} is well-defined for all $\ga\in\Ga$, if, say, $a$ has a bounded support. It is worth noting, that, regardless of $a$, \eqref{Mgen} is defined pointwise, for $\ga\in\Ga_0$. This is sufficient to consider
\begin{equation}\label{Lhatdef}
  (\widehat{L}G)(\eta):=(K^{-1} LK G)(\eta), \quad \eta\in\Ga_0, G\in\Bbs.
\end{equation}
By results of \cite{FKO2009} and \cite[Proposition~3.1]{FKKZ2014}, one has that, for any $G\in\Bbs$, $\eta\in\Ga_0$,
\begin{equation}
\label{Lhat}
\begin{split}
  (\widehat{L}G)(\eta)=&-E^a(\eta)G(\eta)-\sum_{x\in\eta}\Bigl(\sum_{y\in\eta\setminus x} a(x-y)\Bigr)G(\eta\setminus x) \\
  &-m\sum_{\xi \subset \eta }G(\xi )\sum_{x\in
\xi }e^{-E^{\phi }\left( x,\xi \setminus x\right) }e_{\lambda }\bigl(
e^{-\phi \left( x-\cdot \right) }-1,\eta \setminus \xi \bigr)\\&+\la\int_\X G(\eta\cup x)\,dx,
\end{split}
\end{equation}
where
\begin{align}\label{Energy}
  E^a(\eta)&:=\sum_{x\in\eta}\sum_{y\in\ga\setminus x}a(x-y), \quad \eta\in\Ga_0,\\
  E^\phi(x,\eta\setminus x)&:=\sum_{y\in\ga\setminus x}\phi(x-y), \quad \eta\in\Ga_0, x\in\eta,
\end{align}
and, for any measurable $f:\X\to\R$,
\begin{equation}\label{LPexp}
e_\la (f,\eta ):=\prod_{x\in \eta }f(x) ,\ \eta \in \Ga
_0\!\setminus\!\{\emptyset\},\quad e_\la (f,\emptyset ):=1.
\end{equation}

The mapping $L$ defines an evolution of measures in $\Mfm$. Namely, for a given $\mu_0\in\Mfm$, consider the initial value problem
\begin{equation}\label{BKE}
  \begin{cases}
  \frac{d}{dt}\int_\Ga F(\ga)\,d\mu_t(\ga)=\int_\Ga (LF)(\ga)\,d\mu_t(\ga), \quad t>0,\\[2mm]
  \mu_t\bigr\rvert_{t=0}=\mu_0,
  \end{cases}
\end{equation}
which should hold for any $F\in K\bigl(\Bbs\bigr)$ such that the right hand side of \eqref{BKE} is well-defined. The equation \eqref{BKE} may be rewritten in terms of the correlation functions $k_t:=k_{\mu_t}$ of measures $\mu_t\in\Mfm$, provided that they all do exist. Namely, one has
\begin{equation}\label{BKE-desc}
  \begin{cases}
  \frac{d}{dt}\int_{\Ga_0} G(\eta)k_t(\eta)\,d\eta=\int_{\Ga_0} (\widehat{L}G)(\eta)k_t(\eta)\,d\eta, \quad t>0,\\[2mm]
  k_t\bigr\rvert_{t=0}=k_0=k_{\mu_0}.
  \end{cases}
\end{equation}
The latter equation will be the main object of our interest. For relations between solutions to \eqref{BKE-desc} and \eqref{BKE} see, e.g., \cite{FKK2011a}.
One can rewrite \eqref{BKE-desc} in the ``strong'' form:
\begin{equation}\label{QFP}
  \frac{\partial}{\partial t} k_t(\eta) = (L^\triangle k_t)(\eta), \quad \eta\in\Ga_0, t>0,
\end{equation}
where the linear mapping $L^\triangle$ is defined via the duality
\begin{equation}\label{duality}
  \int_{\Ga_0} (\widehat{L}G)(\eta)k(\eta)\,d\eta=
    \int_{\Ga_0} G(\eta)(L^\triangle k)(\eta)\,d\eta,
\end{equation}
for $G,k\in\Bbs$, and it is extended to the linear operator in (a scale of) Banach spaces by the constructions below. By, e.g., \cite{FKO2009} and \cite{FKKZ2014}, one has that
\begin{equation}\label{Ltr}
  \begin{split}
  (L^\triangle k)(\eta)=&-E^a(\eta)k(\eta)-\sum_{y\in\eta} \int_\X a(x-y)k(\eta\cup x)\,dx \\&-m\sum_{x\in \eta
}e^{-E^{\phi }\left( x,\eta \setminus x\right) }\int_{\Gamma _{0}}k\left(
\eta \cup \xi \right) e_{\lambda }\bigl( e^{-\phi \left( x-\cdot \right)
}-1,\xi \bigr) d \xi \\& +\la\sum_{x\in\eta}k(\eta\setminus x), \quad \eta\in\Ga_0.
\end{split}
\end{equation}

We consider the following scale of Banach spaces:
\begin{equation}\label{Kc}
  \K_\a:=\bigl\{ k:\Ga_0\to\R \bigm\vert k(\eta) \a ^{-|\eta|}\in L^\infty(\Ga_0,d\eta)\bigr\}, \quad \a >1,
\end{equation}
with the norms given by
\begin{equation}\label{norm}
  \lVert k\rVert_\a :=\esssup_{\eta\in\Ga_0}|k(\eta)|\a ^{-|\eta|}.
\end{equation}
For the motivation, see, e.g.,\cite{FKK2011a,FKK2014}. It is easy to see, that $\{\K_\a \}_{\a >1}$ satisfies to Assumption~\ref{as:scale} with $\al=1$, $I=(1,\infty)$.

We set, for $\eta\in\Ga_0$,
\[
  (Ak)(\eta) =-E^a(\eta) k(\eta), \qquad
  (Zk)(\eta) =(L^\triangle k)-(Ak)(\eta).
\]

\begin{proposition}\label{existdyn}
The linear mappings $A$ and $Z$ satisfy Assumptions~\ref{as:semgen} and~\ref{as:Ovsgen}, correspondingly.
\end{proposition}
\begin{proof}
  The operator $A$ with the maximal domain
  \begin{equation}\label{maxdom}
    \mathcal{D}_\a:=\bigl\{k\in\K_\a \bigm| Ek\in\K_\a \bigr\},
 \end{equation}
 naturally, generates the semigroup
  \begin{equation}\label{semigr}
    (S(t)k)(\eta)=e^{-tE(\eta)}k(\eta),\quad\eta\in\Ga_0,
  \end{equation}
  in any $\K_\a $, $\a >1$. However, this semigroup is not a $C_0$ one. Indeed, for $k_\a (\eta):=\a^{-|\eta|}$, one has
  \[
  \lVert S(t)k_\a - k_\a \rVert_\a =\esssup_{\eta\in\Ga_0}\bigl\lvert e^{-tE(\eta)}-1\bigr\rvert=1\not\to0, \quad t\to0.
  \]
  Therefore, one should use the technique of the $\odot$-dual semigroups; for details see, e.g., \cite{vNee1992,EN2000}. Namely, we consider the $C _0$-semigroup given by the same expression \eqref{semigr}, but considered in the space $\L_\a := L^1(\Ga_0, \a ^{|\eta|}\,d\eta)$. Then $S(t)$ is dual to that semigroup in the dual space $\K_\a $ (where duality is realized by \eqref{duality}). Then, the space $C_\a:=\overline{\mathcal{D}_\a}$ (the closure is in the norm of $\K_\a$) is $S(t)$-invariant and the restriction $S_\a(t):=S(t)\upharpoonright_{C_\a}$ consists a $C_0$-semigroup there. The generator of $S_\a(t)$ will be the part of $A$, i.e. $(A,D_\a)$, where $D_\a=\{k\in C_\a \mid Ak\in C_\a\}$, cf.~\cite{FKK2011a}. Hence $D_\a$ is $S(t)$-invariant as well. It should be stressed also that $\K_{\a'}\subset D_\a$, for any $1<\a'<\a$. Indeed, for a $k\in\K_{\a'}$,
\begin{equation}\label{sqest}
   \a^{-|\eta|}E^a(\eta)|k(\eta)|\leq \lVert k\rVert_{\a'} \bar a |\eta|^2\Bigl(\frac{\a'}{\a}\Bigr)^{|\eta|}\leq\frac{4\lVert k\rVert_{\a'} \bar a }{e^2\ln^2\frac{\a'}{\a} },
\end{equation}
  where we used that $\sup_{r>0}r^2q^r=4/(e\ln q)^2$, for $q\in(0,1)$. Since $|S(t)k|\leq |k|$ pointwise, the space $\K_{\a'}$ is also $S(t)$-invariant. From these arguments we easily get that $A$ satisfies Assumption~\ref{as:semgen} with $\nu=1$, $\omega=0$.

  Next, let us denote
  \[
  (Z^{(1)}k)(\eta):=-\sum_{y\in\eta} \int_\X a(x-y)k(\eta\cup x)\,dx.
  \]
  Then, for any $1<\a'<\a''<\a^*$ and for any $k\in\K_{\a'}$,
  \[
  (\a'')^{-|\eta|}|(Z^{(1)}k)(\eta)|\leq \langle a\rangle \a' \lVert k\rVert_{\a'}|\eta| \Bigl(\frac{\a'}{\a''}\Bigr)^{|\eta|}\leq \langle a\rangle \a' \lVert k\rVert_{\a'}
  \frac{1}{-e\ln\frac{\a'}{\a''}}\leq
  \frac{\langle a\rangle (\a^*)^2 \lVert k\rVert_{\a'}}{e(\a''-\a')},
  \]
  where we used that $\sup_{r>0}rq^r=1/(e\ln q)$, for $q\in(0,1)$, and $\ln \a''-\ln\a' =\frac{1}{\tilde{\a}}(c''-c')$, for some $\tilde{\a}\in(\a',\a'')$.

  The similar estimate for $ (Z^{(2)}k)(\eta):= (Zk)(\eta)- (Z^{(1)}k)(\eta)$ was obtained in \cite[Proposition~3.2]{FKKZ2014}. Combining these results, one gets that $Z$ satisfies Assumption~\ref{as:Ovsgen}, with
  \[
  M(\a^*)=\frac{1}{e}\bigl( \langle a\rangle (\a^*)^2 +\a^* me^{\langle\phi\rangle\a^*}+\a^*\la\bigr).
  \]
  (To be more precise, we used here the estimate $\int_\X (1-e^{-\phi(x)})\,dx\leq \langle\phi\rangle$, to simplify the expression from \cite[Proposition~3.2]{FKKZ2014}.)
\end{proof}

Consider now the so-called Vlasov scaling of the dynamics above, see, e.g., \cite{FKK2010a,FKK2011a}. Namely, for an $\eps>0$, we denote by $L_\eps$ the operator \eqref{Mgen} with $\eps a(\cdot)$, $\eps\phi(\cdot)$, $\eps^{-1}\la$ in place of $a(\cdot)$, $\phi(\cdot)$, $\la$, respectively. Then, one can construct $L_\eps^\triangle$ in the same way as above. We~set~also
\begin{equation}\label{Lren}
  (L^\triangle_{\eps,\mathrm{ren}}k)(\eta):=\eps^{|\eta|}L^\triangle_{\eps} \eps^{-|\eta|} k(\eta).
\end{equation}
Directly from \eqref{Ltr}, one gets
\begin{equation}\label{Ltrren}
  \begin{split}
  (L^\triangle_{\eps,\mathrm{ren}} k)(\eta)=&-\eps E^a(\eta)k(\eta)-\sum_{y\in\eta} \int_\X a(x-y)k(\eta\cup x)\,dx \\&-m\sum_{x\in \eta
}e^{-\eps E^{\phi }\left( x,\eta \setminus x\right) }\int_{\Gamma _{0}}k\left(
\eta \cup \xi \right) e_{\lambda }\Bigl( \frac{e^{-\eps\phi \left( x-\cdot \right)
}-1}{\eps},\xi \Bigr) d \xi \\& +\la\sum_{x\in\eta}k(\eta\setminus x), \quad \eta\in\Ga_0.
\end{split}
\end{equation}

We denote, for $\eps>0$,
\[
  (A_\eps k)(\eta):=-\eps E^a(\eta)k(\eta), \qquad
  (Z_\eps k)(\eta):=(L^\triangle_{\eps,\mathrm{ren}} k)(\eta)- (A_\eps k)(\eta),
\]
and we set, naturally, $(A_0 k)(\eta):=0$ and
\begin{align*}
(Z_0 k)(\eta):=&-\sum_{y\in\eta} \int_\X a(x-y)k(\eta\cup x)\,dx \\&-m\sum_{x\in \eta
}\int_{\Gamma _{0}}k\left(
\eta \cup \xi \right) e_{\lambda }\bigl( -\phi \left( x-\cdot \right),\xi \bigr) d \xi +\la\sum_{x\in\eta}k(\eta\setminus x).
\end{align*}

\begin{proposition}
The linear mappings $A_\eps$ and $Z_\eps$, $\eps\geq0$ satisfy Assumptions~\ref{as:semgen-eps} and~\ref{as:Ovsgen-eps}, correspondingly. Moreover, the conditions \eqref{convsg}, \eqref{convZ}, \eqref{seq0} of Theorem~\ref{thm:conv} hold.
\end{proposition}
\begin{proof}
The operators $A_\eps$ satisfy Assumption~\ref{as:semgen-eps} by the same arguments as in the proof of Proposition~\ref{existdyn} (independently on $\eps\geq0$). The operators $Z_\eps$ satisfy Assumption~\ref{as:Ovsgen-eps} by the estimation for $Z^{(1)}$ in the proof of Proposition~\ref{existdyn} and by \cite[Proposition~4.2]{FKKZ2014} for $Z_\eps-Z^{(1)}$, $\eps\geq0$.

Next, in the notations of Theorem~\ref{thm:conv}, for any $1<\a'<\a''<\a^*$ and for any $k\in\a'$,
\begin{align*}
\lVert S_\eps(t) k - S_0 (t)k\rVert_{\a''}&=\esssup_{\eta\in\Ga_0} \ (\a'')^{-|\eta|} \bigl\lvert e^{-t\eps  E^a(\eta)}-1\bigr| |k(\eta)|\\&\leq \lVert k \rVert_{\a'} (s+T) \eps \esssup_{\eta\in\Ga_0} E^a(\eta)\Bigl(\frac{\a'}{\a}\Bigr)^{|\eta|},
\end{align*}
that implies \eqref{convsg}, by using the same estimate as in \eqref{sqest}.

Finally,
\begin{align*}
  (Z_\eps k)(\eta)-(Z_0 k)(\eta)=&-m\sum_{x\in \eta
}e^{-\eps E^{\phi }\left( x,\eta \setminus x\right) }\int_{\Gamma _{0}}k\left(
\eta \cup \xi \right) e_{\lambda }\Bigl( \frac{e^{-\eps\phi \left( x-\cdot \right)
}-1}{\eps},\xi \Bigr) d \xi\\&+m\sum_{x\in \eta
}\int_{\Gamma _{0}}k\left(
\eta \cup \xi \right) e_{\lambda }\bigl( -\phi \left( x-\cdot \right),\xi \bigr) d \xi,
\end{align*}
and the estimate \eqref{convZ} (with $r=2$) was proved in \cite[Proposition~4.6]{FKKZ2014}.
\end{proof}

Suppose now, for simplicity, that $k_{0,\eps}:=k_0\in\K_{\alpha_0}$, $\alpha_0>1$, $\eps>0$. Then, by Theorem~\ref{thm:conv}, the solutions to the equation
\begin{equation}\label{QFP-eps}
  \frac{\partial}{\partial t} k_{t,\eps}(\eta) = (L^\triangle_{\eps,\mathrm{ren}} k_{t,\eps})(\eta), \quad \eta\in\Ga_0, t\in(0,T),
\end{equation}
converges in any $\K_{\a^*}$ with $\a^*>\a_0$ to the solution to the equation
\begin{equation}\label{QFP-V}
  \frac{\partial}{\partial t} k_t(\eta) = (Z_0 k_t)(\eta), \quad \eta\in\Ga_0, \ t\in(0,T),
\end{equation}
uniformly on any $[0,\Tau]\subset(0,T)$, where $T=T(\a_0,\a^*)$.

The limiting equation \eqref{QFP-V} has the following key-property: if $k_0(\eta)=e_\la(\rho_0,\eta)$, for a function $\rho_0\in L^\infty(\X,dx)$ then one can find a (unique) solution to \eqref{QFP-V} of the same form: $k_t(\eta)=e_\la(\rho_t,\eta)$. To show this, note that
\[
\frac{\partial}{\partial t}e_\la(\rho_t,\eta)=\sum_{x\in\eta}\rho_t(x)e_\la(\rho_t,\eta\setminus x)
\]
and
\begin{align*}
(Z_0 e_\la(\rho_t))(\eta):=&-\sum_{x\in\eta} \rho_t(x) \int_\X a(x-y)\rho_t(y)\,dy e_\la(\rho_t,\eta\setminus x) \\&-m\sum_{x\in \eta
}\rho_t(x) e_\la(\rho_t,\eta\setminus x) \int_{\Gamma _{0}}e_\la(\rho_t,\xi)  e_{\lambda }\bigl( -\phi \left( x-\cdot \right),\xi \bigr) d \xi \\&+\la\sum_{x\in\eta}e_\la(\rho_t,\eta\setminus x).
\end{align*}
By \eqref{LebPois}, we have, for any $f\in L^1(\X,dx)$,
\[
\int_{\Ga_0} e_\la(f,\eta)\,d\eta=\exp\Bigl\{\int_\X f(x)\,dx\Bigr\}.
\]
Therefore, $k_t(\eta)=e_\la(\rho_t,\eta)$ indeed solves \eqref{QFP-V}, provided that $\rho_t$ is a unique solution to the following equation
\begin{equation}\label{KE}
  \frac{\partial}{\partial t}\rho_t(x)=-\rho_t(x)(a*\rho_t)(x)-m\rho_t(x)e^{-(\phi*\rho_t)(x)}+\la,
\end{equation}
in the space $L^\infty(\X,dx)$ (at least on $(0,T)$).

The existence and uniqueness of nonnegative solutions to \eqref{KE} may be done using the same approaches as in \cite{FKKZ2014,FKKozK2014}. We will realise this in a sequel paper.

\begin{remark}\label{statpoints}
  It is worth noting that the equation \eqref{KE} may have one or three positive stationary solutions depending on values of the parameters. Indeed, if $\rho_t(x)\equiv\rho>0$ is a stationary solution to \eqref{KE}, then $\la=\langle a\rangle \rho^2+m\rho \exp(-\langle\phi\rangle \rho)$. Denote $x=\langle\phi\rangle \rho>0$, $c=\frac{\la \langle\phi\rangle}{m} $, $b=\frac{\langle a\rangle}{m\langle\phi\rangle} $; then we will get $xe^{-x}+bx^2=c$.
  The function $f(x)=xe^{-x}+bx^2$, $x\geq0$, may have zero or two points of local extremum. Indeed,
  $f'(x)=0$ yields $2bx=(x-1)e^{-x}$. The function $g(x)=(x-1)e^{-x}$, $x\geq0$, has the derivative $g'(x)=(2-x)e^{-x}$ and hence $g$ increases from $-1$ to $e^{-2}$ on $(0,2)$ and decreases for $x>2$. The tangent line to the graph of $y=g(x)$ at a point $(x_0,g(x_0))$ which passes through the origin has the equation $y-g(x_0)=g'(x_0)(x-x_0)$, and thus $x=y=0$ yields
  \[
  -(x_0-1)e^{-x_0}=(2-x_0)e^{-x_0}(-x_0), \quad x_0^2-x_0-1=0,\quad x_0=\frac{1+\sqrt{5}}{2}>0.
  \]
  As a result, if $2b<g'(x_0)$, i.e. if
  \[
  \frac{\langle a\rangle}{m\langle \phi\rangle}<\frac{3-\sqrt{5}}{4}\exp\Bigl(-\frac{1+\sqrt{5}}{2}\Bigr),
  \]
  then the function $f(x)$ has two points of local extremum and, therefore, there exists $\la$ (and thus $c$) such that the equation $f(x)=c$ has three solutions. For $2b\geq g'(x_0)$, it will have one solution only, for any $\la>0$.
\end{remark}

\end{document}